\documentclass[10pt]{amsart}
\usepackage[run=2]{auto-pst-pdf}
\usepackage{amsmath}
\usepackage{amssymb}
\usepackage{amscd}
\usepackage{comment}
\usepackage{url}
\usepackage[all]{xypic}
\usepackage{tikz}
\usepackage{graphicx}

\newcommand{\bp}{\begin{pmatrix}}
\newcommand{\ep}{\end{pmatrix}}
\newcommand{\be}{\begin{equation}}
\newcommand{\ee}{\end{equation}}
\newcommand{\ol}[1]{\overline{#1}}
\newcommand{\e}{\varepsilon}

\numberwithin{equation}{section}

\theoremstyle{plain}
\newtheorem{theorem}[equation]{Theorem}
\newtheorem*{theorem*}{Theorem}
\newtheorem{lemma}[equation]{Lemma}
\newtheorem{proposition}[equation]{Proposition}

\newtheorem{corollary}[equation]{Corollary}

\theoremstyle{definition}
\newtheorem{example}[equation]{Example}

\newtheorem{definition}[equation]{Definition}

\theoremstyle{remark}
\newtheorem{remark}[equation]{Remark}

\numberwithin{equation}{section}

\newtheorem*{ack}{Acknowledgements}
\newcommand{\intfrac}[2]{\genfrac{\lfloor}{\rfloor}{}{1}{#1}{#2}}

\def\Z{\mathbb Z}
\def\R{\mathbb R}
\def\Q{\mathbb Q}
\def\C{\mathbb C}

\def\A{\mathcal{A}}
\def\O{\Omega}
\def\wt#1{\widetilde{#1}}
\def\wC{\wt{C}}
\def\p{\partial}

\def\sm{\setminus}

\def\kc{K_{cen}}
\def\kt{K_{tot}}
\def\kz#1{K_{z_{#1}}}
\def\level{double point count}
\newcommand{\sss}{\ifmmode{{\mathfrak s}}\else{${\mathfrak s}$\ }\fi}
\newcommand{\sst}{\ifmmode{{\mathfrak t}}\else{${\mathfrak t}$\ }\fi}
\def\spinc{Spin$^c$}

\DeclareMathOperator\coker{coker}
\DeclareMathOperator\Hom{Hom}

\title[Deformations of singular points]{Semigroups, 
 $\pmb{d}$--invariants 
 and deformations of cuspidal singular points of plane curves}
\author{Maciej Borodzik}
\address{Institute of Mathematics, University of Warsaw, ul. Banacha 2,
02-097 Warsaw, Poland}
\email{mcboro@mimuw.edu.pl}
\thanks{The first author was supported by  Polish OPUS grant No 2012/05/B/ST1/03195}
\thanks{The second author was supported by
 National Science Foundation   Grant  1007196.}

\author{Charles Livingston}
\address{Department of Mathematics, Indiana University, Bloomington, IN 47405}
\email{livingst@indiana.edu}

\makeatletter
\@namedef{subjclassname@2010}{\textup{2010} Mathematics Subject Classification}
\makeatother

\begin{document}
\begin{abstract}
We study $\delta$--constant deformations of plane curve singularities from topological point of view. We introduce a notion of a positively self--intersecting
concordance, which is a topological counterpart of a $\delta$--constant deformation in singularity theory. We then apply methods from
the Heegaard--Floer theory. As a main result we give a purely topological proof of the semicontinuity property for semigroups of singular points
of plane curves under $\delta$--constant deformation. 
Using the same approach we give also a knot--theoretical result concerning minimal unknotting sequences of torus knots.\end{abstract}
\maketitle

\section{Cuspidal singular points and $\delta$--constant deformations}\label{sec:first}

\subsection{Introduction}
Let $K_0$ and $K_1$ be knots in $S^3$. A \emph{positively self--intersecting concordance} from $K_0$ to $K_1$ 
is a smoothly immersed annulus $A\subset S^3\times[0,1]$ with boundary $K_1\times\{1\}\sqcup -K_0\times\{0\}$. The annulus is only allowed
to have standard double points as singularities, and the intersection index
of the two branches of $A$ at each double point must be $+1$. The \emph{double point count} of a positively self--intersecting concordance is the number $p(A)$
of the double points of the annulus $A$.

From a knot theoretical point of view, the notion of a positively self--intersecting concordance encompasses two phenomena:  standard concordance and unknotting moves. For example, one can prove that a positively self--intersecting concordance from $K_0$ to $K_1$ with a double point
count one exists if and only if there exist knots $K_0'$ and $K_1'$ such that $K_0$ is concordant to $K_0'$; $K_1$ is concordant to $K_1'$; and
$K_1'$ is obtained from $K_0'$ by changing one negative crossing to a positive.

The main motivation for introducing the notion of a positively self--intersecting concordance comes from singularity theory. As we show in detail
in Section~\ref{sec:psic}, a $\delta$--constant deformation (see Definition~\ref{def:deltaconst}) gives rise to a positively self--intersecting concordance
between the corresponding links of singular points (in our applications, these links have one component, so they are actually knots). 
The double point count in this case is easily seen to be 
the difference of the three--genera of these knots. The topological
notion of a positively
self--intersecting concordance seems to be a natural topological counterpart of the algebraic notion of the $\delta$--constant deformation.

Coming back to topology, starting from a positively self--intersecting concordance from a knot $K_0$ to a knot $K_1$, with double point count $p$, we can construct, for
any integer $q$, a smooth four manifold $W_q$.  This manifold is a cobordism between  $(q+4p)$--surgery on $K_1$, which we denote $S^3_{q+4p}(K_1)$,  and $q$--surgery on $K_0$, $S^3_q(K_0)$.   Moreover, if $q(q+4p)>0$, then
the intersection form on $H_2(W_q)$ is negative definite. This means that we can apply Heegaard--Floer theory   to compare corresponding $d$--invariants
of $S^3_{q+4p}(K_1)$ and $S^3_{q}(K_0)$ to actually obstruct the existence of a positively self--intersecting concordance from $K_0$ to $K_1$
with double point count $p$.
The $d$--invariants of surgeries on knots are usually hard to compute, but if $K_0$ and $K_1$ are algebraic knots   and $q$ is a large positive number, then these $d$--invariants can be explicitly expressed in terms of the coefficients
of the Alexander polynomials of $K_0$ and $K_1$.  (Similarly, computations are simplified for  so called $L$--space knots, that is,
knots that admit a positive surgery which
is an $L$--space; but we will not need the general theory of $L$--spaces and $L$--space knots here.)
The computations are
at the beginning quite involved, but a remarkable simplification occurs, leading to the following result, which is the main technical result
of the present article.

\begin{theorem*}[Theorem~\ref{prop:dinv}]
Suppose $K_0$ and $K_1$ are connected sums of algebraic knots and there exists a positively self--intersecting concordance from $K_0$ to $K_1$
with \level~$p$. Then for any $m\in\Z$ we have
\[I_{{K_1}}(m+g(K_1)+p)\le I_{K_0}(m+g(K_0)),\]
where $I$ is the gap function defined in Section~\ref{sec:surg} and $g(K)$ denotes the three genus of the knot $K$.
\end{theorem*}
\begin{remark}
The result can be easily extended for all connected sums of $L$--space knots with essentially the same proof, 
but the meaning of the gap function is  less geometric and therefore we focus on algebraic knots.
\end{remark}
From this result we draw two consequences. The first one deals with unknotting sequences of torus knots. 
To the best of our knowledge this is the first known proof of this result. The statement can be seen as 
a topological counterpart of the semicontinuity
of the multiplicity of a singular point in singularity theory. 

\begin{theorem*}[On minimal unknotting sequences of torus knots, Theorem~\ref{thm:unknot}]
If  the  torus knot $T(a,b)$ with $0< a<b$ appears in some minimal unknotting sequence for the torus knot $T(c,d)$  with $0< c<d$, then $a \le c$.
\end{theorem*}

\smallskip
The next result is a semicontinuity property for semigroups for $\delta$--constant deformations. This is the main result of the present article.

\begin{theorem*}[Semicontinuity of semigroups, Theorem~\ref{thm:main}]
Assume that there exists a $\delta$--constant deformation of a cuspidal singular point $z_0$ such that the general fiber
has among its singularities cuspidal singular points $z_1,\ldots,z_n$. Let $S_0,S_1,\ldots,S_n$, be the corresponding semigroups.
Then for any $m\ge 0$ and for any nonnegative $m_1,\ldots,m_n$ such that $m_1+\ldots+m_n=m$, we have
\[\# S_0\cap [0,m)\le \sum_{j=1}^n\# S_j\cap [0,m_j).\]
\end{theorem*}

Theorem~\ref{thm:main} can be regarded as a generalization of semicontinuity of the multiplicity, a result  first proved by Gorsky and N\'emethi in \cite{GN}. They did not make an assumption that the deformation is $\delta$--constant,
but they proved the result only for $n=1$. In Section~\ref{sec:many} we will present examples of cases not covered by Gorsky and N\'emethi.  On the other hand,
it is almost certain that there exists a proof of the semicontinuity property for many singular points (that is, $n$ can be arbitrary) 
and for general, not only $\delta$--constant,
deformations.  
What is new and most important about our proof is that  it works in a smooth category, that is, this is a topological proof; 
 Gorsky and N\'emethi prove their result
using tools from algebraic geometry which are not easily translated into the topological world. 
In our proof, we identify the key topological properties entailed by the existence of a $\delta$--constant deformation and then proceed using tools solely from smooth topology.

In particular we never rely on the fact that the positively self--intersecting concordance is realized by a complex algebraic annulus. Clearly, the
condition that all the double points of an annulus are positive is automatically satisfied if the annulus has a complex structure; but the existence of a complex
structure is definitely a stronger condition that the assumption of only positivity of intersections.

On the other hand, the smooth structure of the positively self--intersecting concordance is really used. 
As we discuss in Remark~\ref{rem:smooth}, Theorem~\ref{prop:dinv} is in general false if we replace
`smooth' by `topological locally flat' in the definition of the positively self--intersecting concordance. To the best knowledge of the authors,
the semigroup semicontinuity is the first obstruction for deformations of singular points of plane curves that works precisely in the smooth category.

The structure of the paper is the following. First, in Section~\ref{sec:first}, 
we recall various notions  from topology and from singularity theory. Experts might find this material to be elementary, but we want
this article to be accessible for specialists   in topology as well as for specialists  in singularity theory. 

Next we move to positively self--intersecting concordances in Section~\ref{sec:positively}. We provide a construction of the manifold $W_q$ mentioned above
and study its properties. In particular, we prove Theorem~\ref{prop:dinv} and Theorem~\ref{thm:main}. In Section~\ref{sec:baader} we
prove Theorem~\ref{thm:unknot}.

Section~\ref{examplesection} contains examples of application of Theorem~\ref{thm:main}. We compare our criterion with other criteria from
deformation theory, especially with semicontinuity of the spectrum.

\begin{ack}
We are grateful to Sebastian Baader, Eugene Gorsky, Andr\'as N\'emethi and Peter Ozsv\'ath 
for fruitful discussions. We would like also to thank to Javier Fernandez de Bobadilla
for his interest in our work and for spotting a misprint in the main theorem. The first author would like to thank   Indiana University
for its hospitality.
\end{ack}

\subsection{Genus and normalization}\label{sec:genus}

We recall the notion of  the genus of a singular curve. The general theory can be found in \cite[Section II.8, Applications]{Hart}; 
we will present  the  necessary background.

For a real two-dimensional compact surface $\Sigma$ with boundary, we define the genus by the formula
\begin{equation}\label{eq:genus}
2b_0(\Sigma)-2g(\Sigma)=\chi(\Sigma)+b_0(\partial\Sigma),
\end{equation}
where $b_0$ stands for the Betti number. 

Suppose $C$ is an algebraic curve.

\begin{lemma}[see \expandafter{\cite[Section I.1.9]{GLS} or \cite[I, Exercise 3.17]{Hart}}]
There exists a Riemann surface $\Sigma$, a holomorphic surjective
map $\phi\colon\Sigma\to C$, and a finite set $A\subset\Sigma$, such that $\phi$ restricted
to $\Sigma\setminus A$ is one-to-one. 
\end{lemma}
The surface $\Sigma$ is called the \emph{normalization} of $C$ and the map $\phi$ is called the \emph{normalization map}.
The pair $(\Sigma,\phi)$ is uniquely defined up to biholomorphism. The main property of $\phi$ is that for any singular point $z\in C$ with $k$ branches,
the inverse image under $\phi$ of a neighborhood of $z$ is a disjoint union of $k$ disks. One can mimic the normalization of plane algebraic curves
in the  smooth or topological category as well.

\begin{definition}\label{genusdef}
For an algebraic curve $C$, the \emph{genus} of $g(C)$ is defined as the genus of $\Sigma$.
\end{definition}
In algebraic geometry, $g(C)$ is often called the geometric genus and is sometimes denoted $p_g(C)$, for example in \cite{Hart}.

\subsection{Cuspidal singularities and their semigroups}\label{sec:semigroups}
We now  review the  necessary background of cuspidal singularities and semigroups. Good references include~\cite{BK,EN,GLS,Wa}.
Here and in what follows, $B(z,r)$  shall denote a ball in $\C^2$ of
radius $r$ and center $z$.

By an  \emph{isolated plane curve singularity} we mean a pair $(C,z)$, where $C$ is a complex curve in $\C^2$ and $z$ is a point on $C$, such
that $C$ is smooth at all points sufficiently close to $z$, with the exception of $z$ itself. 
For a sufficiently small $r>0$, $C$ intersects the sphere $\partial B(z,r)\subset\C^2$ transversally along
a link $L$ and $C\cap B(z,r)$ is topologically a cone on this link.  The link $L$ will be called the \emph{link of the singularity}.
 
\begin{definition}
A singularity is called \emph{cuspidal} or \emph{unibranched} if $L$ is connected; that is, if $L$ is a knot.
\end{definition}
\begin{remark}
Strictly speaking, in the pair $(C,z)$ above, $C$ should be a \emph{germ} of a complex curve. In that case,  
each time  a geometric construction involves $C$ (for instance, defining the  link of a singularity), one uses a representative of $C$.  
This distinction, while technically important, does not affect the arguments.
\end{remark}

We shall consider singular points up to so-called \emph{topological equivalence}. Two singularities are said to be  \emph{topologically equivalent}
if their links are isotopic. We refer to \cite[Section I.3.4]{GLS} or \cite[Section 8.3, page 414]{BK} for more details concerning this definition.

It is a result going back to Burau and Zariski~\cite{Bu,Za}  that any cuspidal singularity is topologically equivalent to a singularity
which is locally  parametrized by $x=t^p$, $y=t^{q_1}+\ldots+t^{q_n}$, where $p<q_1<\ldots<q_n$, $\gcd(p,q_1,\ldots,q_{k-1})$ does not divide $q_k$, and
$\gcd(p,q_1,\ldots,q_n)=1$. We refer to the data $(p;q_1,\ldots,q_n)$ as the \emph{characteristic sequence}. There is a one-to-one
correspondence between characteristic sequences satisfying constraints as above and topological types of singular points.
In particular, the characteristic sequence determines the link of the singularity.
We refer 
to \cite{EN} for a detailed description.

For a cuspidal singularity $(C,z)$, we consider a normalization of $C\cap B(z,r)$. We can choose it to be  an analytic map $\gamma\colon V\to B(z,r)\subset\C^2$,
where $V$ is an open neighborhood of $0\in\C$, such that $\gamma(0)=z$ and $\gamma$ is a homeomorphism between $V$ and $C\cap B(z,r)$ (in 
\cite[Section 2.3]{Wa} such map is called a \emph{good parametrization}). We can write $\gamma=(\phi,\psi)$ in coordinates of $\C^2$.
Then $\gamma$ induces a map between the rings of power series
\[\gamma^*\colon \C[[x,y]]\to\C[[t]],\]
by the substitution $f(x,y)\mapsto f(\phi(t)-z_1,\psi(t)-z_2)$, where $(z_1,z_2)$ denote the coordinates of the singular point $z$. The map
$\operatorname{ord}\colon \C[[t]]\to\Z_{\ge 0}\cup\{\infty\}$ maps a power series in one variable to its order at $0$ (by convention, if  $f$ vanishes
identically on $C$, its order is defined to be $+\infty$). The set of non-negative integers (excluding infinity) belonging to
the image of $\C[[x,y]]$ under the composition $\operatorname{ord}\circ\gamma^*$,
\begin{equation}\label{eq:G}
S=\operatorname{ord}(\gamma^*(\C[[x,y]])),
\end{equation}
is easily seen to be a semigroup  (see \cite[Section 4.3]{Wa}).

\begin{definition}
The semigroup $S$ is called the semigroup of the cuspidal singularity.
\end{definition}

The next result can be found, for example, in \cite[Chapter 4]{Wa}.
\begin{lemma}\label{lem:semiprop}
Let $S$ be the semigroup of a cuspidal singularity  with   characteristic sequence $(p;q_1,\ldots,q_n)$.
\begin{enumerate}
\item[(a)]   $S$ depends only on the topological equivalence class; in particular, it is determined by the characteristic sequence.
\item[(b)]  $S$ has a minimal generating set with  $n+1$ elements.
\item[(c)]  The cardinality of the set $\Z_{\ge 0}\sm S$ is equal to $\mu/2$, where $\mu$ is the Milnor number of the singularity (for a cuspidal singularity
the Milnor number is equal to twice the three genus of the link of the singularity).
\item[(d)]  The maximal element $k$ which does not belong to the semigroup is $\mu-1$. Furthermore, for any $j=0,\ldots,\mu-1$, exactly
one of the values $j$ or $\mu-j-1$ belongs to the semigroup.
\end{enumerate}
\end{lemma}
\begin{example}
Suppose that the singularity is given by $\{x^p-y^q=0\}$, where $p$ and $q$ are coprime integers. Then its link is the  torus knot $T(p,q)$.
The map $t\to(t^q,t^p)$ is a normalization of the singularity, hence the image $\gamma^*(\C[[x,y]])$ is $\C[[t^p,t^q]]\subset\C[[t]]$.
This semigroup $S$ is generated by $p$ and $q$. The maximal element not belonging to the semigroup
 can be shown to be    $(p-1)(q-1)-1$.
\end{example}

We point out that semigroups arising from singularities are  special; the last property of Lemma~\ref{lem:semiprop} puts
a very strong restriction on the semigroups that arise in the case that  the number of generators is $3$ or more. For a general semigroup of $\mathbb{N}$ with three
or more generators, computing the value of a maximal element not belonging to a semigroup is a challenging problem; see \cite{RA}.

\subsection{Surgeries on knots and $d$--invariants}\label{sec:surg}

  Heegaard Floer homology was  defined by Ozsv\'ath and Szab\'o. Out of the very long list of their papers, the most
relevant in our applications is \cite{OzSz03}. The methods we are using in this article can be found in more detail in \cite{HHN12, HLR12} or
in \cite{BL}. An overview of  Heegaard Floer homology theory can be found in \cite{OS-introduction,OS-introduction2}. 

Briefly, to any closed oriented three manifold $M$ with
a choice of a \spinc{} structure $\sss$ one associates groups $HF^+(M,\sss)$, $HF^-(M,\sss)$, $HF^{\infty}(M,\sss)$ and $\widehat{HF}(M,\sss)$. The group
$HF^\infty(M,\sss)$ is a $\Z_2[U,U^{-1}]$ module (with $U$ a formal variable), $HF^-$ and $HF^+$ are $\Z_2[U]$ modules. 
There is a natural map $HF^\infty(M,\sss)\to HF^+(M,\sss)$ commuting with multiplication by $U$.  
The $d$--invariant (see \cite{OzSz03}) is the minimal grading of a non-zero element $x\in HF^+(M,\sss)$, 
which is in the image of $HF^\infty$. The fundamental property of $d$--invariants
is the following result proved by Ozsv\'ath and Szab\'o.

\begin{proposition}[see \expandafter{\cite[Section 9]{OzSz03}}]\label{prop:fundbound}
Suppose $(M_1,\sss_1)$ and $(M_2,\sss_2)$ are closed oriented three manifolds such that $b_1(M_1)=b_1(M_2)=0$ and with a choice of a \spinc{} structure $\sss_1$
on $M_1$ and $\sss_2$ on $M_2$. Suppose that $W$ is a  compact oriented
smooth four manifold such that $\partial W=M_2\cup -M_1$ (that is, $W$ is a cobordism between $M_1$ and $M_2$)  and that on $W$ there is a \spinc{} 
structure $\sst$ restricting to $\sss_1$ on $M_1$ and $\sss_2$ on $M_2$. If the intersection form on $H_2(W)$ is \emph{negative definite}, then
\[d(M_2,\sss_2)-d(M_1,\sss_1)\ge \frac{1}{4}\left(c_1^2(\sst)-3\sigma(W)-2\chi(W)\right).\]
\end{proposition}

The main problem with applying this result is that often it is difficult to compute the $d$--invariant. Fortuitously, in our case we are able to do this  because
the knots in question will always be connected sums of algebraic knots.
To state the algorithm, let us introduce some notation.

For a knot $J\subset S^3$,   surgery on $J$ with coefficient $p/q$ will be denoted  $S^3_{p/q}(J)$. In the case that $q=1$ and $p\neq 0$, there is an enumeration of possible
\spinc{} structures on $S^3_p(J)$, explained in detail in Section~\ref{sec:spinc}. For the moment we note that they are enumerated by integers $m\in[-p/2,p/2)$
and the structure corresponding to $m\in[-p/2,p/2)$ will be  denoted by $\sss_m$.
By $g(J)$ we denote the three genus (that is, the Seifert genus) of $J$.

We now suppose that $J=J_1\#\ldots\# J_n$ is a connected sum of algebraic knots. For a knot $J_k$, we denote by $S_k$ the semigroup
of the corresponding singular point and denote by $I_k$ the following function:
\[I_k(m):=\#\{x\in\Z,\ x\ge m,\ x\not\in S_k\}.\]
In \cite{BL} this is called a \emph{gap function}.   As in~\cite{BL}, we define
\[I_J=I_1\diamond I_2\diamond \ldots \diamond I_n.\]
Here \[I\diamond I'(k)=\min_{l\in\Z} I(k-l)+I'(l),\] is the {\it infimum convolution},  defined for any pair of functions  $I,I'\colon\Z\to\Z$ that are  bounded from below.

The following result, proved first by Ozsv\'ath and Szab\'o if $n=1$ (see for instance \cite[Theorem 1.2]{OzSz11}),
gives us a description of $d$--invariants for large surgeries on $J$.
\begin{proposition}[see \cite{BL}]\label{prop:d-compute}
Let $s>2g(J)$ be an integer and let $m\in[-s/2,s/2)$. Then 
\be\label{eq:dofsurg}
d(S^3_s(J),\sss_m)=-2I_J(m+g(J))+\frac{(2m-s)^2+s}{4s}.
\ee
\end{proposition}

\subsection{Deformations}\label{sec:deform}
We adopt a simple definition of a deformation.
In the case of isolated plane curve singularities, this definition is equivalent to a general one (with one-dimensional parameter family), 
given, for example, in \cite[Definition II.1.1]{GLS}.

Let us consider a cuspidal singularity $(C,z)$. There is a neighbourhood $U\subset\C^2$ of $z$ and a map $F\colon U\to\C$
such that $F^{-1}(0)=C\cap U$ and the gradient of $F$ does not vanish at points sufficiently close  but not equal to  $z$. By  the Tougeron theorem, see
\cite[Theorem I.2.5]{Zo}, we can replace $(C,z)$ by an analytically equivalent singularity (see \cite[Definition I.3.30]{GLS}) 
such that $F$ is a polynomial and hence defined on $\C^2$. For convenience, we will assume that $F$ is a polynomial.

So let $F\colon\C^2\to\C$ be a polynomial map such that $(F^{-1}(0),z)$ is a cuspidal singularity. By a \emph{deformation} (sometimes called also
\emph{unfolding})
of this singularity   we mean an analytic family of polynomials $F_s\colon\C^2\to\C$, where $s$ varies in a unit disk $D$ in a complex plane,
$F=F_0$, and for any $s\in D$, the inverse image $F_s^{-1}(0)$ is an  algebraic curve with isolated singularities. In short, 
deforming a singularity can be regarded as varying  the coefficients of the polynomial $F$ in an analytic way (the coefficients are analytic functions of 
the parameter $s$). 

\begin{remark}
In the present article we shall actually use   only  the fact that the coefficients of $F_s$ depend $C^1$--smoothly on the parameter $s$.
\end{remark}

Before we define the $\delta$--constant deformations, let us recall the definition of the $\delta$--invariant. For a singular point $z$ with
Milnor number $\mu$ and $r$ branches, we define $\delta=\frac12(\mu+r-1)$; compare \cite[Section 10]{Milnor-sing}. There are
other equivalent definitions, for example in \cite[Section I.3.4]{GLS}. For cuspidal singularities $\delta$ is equal to the three  genus of the link. For an
ordinary double point (that is, given by $\{(x,y)\in\C^2\colon xy=0\}$), we have $\delta=1$.
 
Consider a deformation and   choose a ball $B(z,r)$ of a small radius  such that $L_0=F_0^{-1}(0)\cap \partial B(z,r)$
is the link of the singularity of $F_0$ at $0$ and $F_0^{-1}(0)\cap B(z,r)$ is a cone over $L_0$. We will assume that $L_0$ has a single branch
so that the singularity is cuspidal. In particular, $F_0^{-1}(0)\cap B(z,r)$ is topologically a disk.

For this given $r$, there exists  a $\rho>0$ such that if $|s|<\rho$, the intersection $F_s^{-1}(0)\cap \partial B(z,r)$ 
is isotopic to $L_0$. Let us choose such $\rho$.

\begin{definition}[see \expandafter{\cite[Section II.2.7]{GLS} or \cite{Te}}]\label{def:deltaconst}
A deformation is called $\delta$--constant if there exists $\rho'\in(0,\rho)$ such that for any $s\in D$ such that $|s|<\rho'$,
the intersection $F_s^{-1}(0)\cap B(z,r)$ has genus 0 (in the sense of Definition~\ref{genusdef}).  
\end{definition}
\begin{example}\label{ex:simpledeformation}
Consider the  deformation $F_s=x^3-y^2+s$ with $z=(0,0)$. For any $r>0$, an apropriate $\rho$ as above, and $0<|s|<\rho$, the inverse 
image $F^{-1}_s(0)\cap B(0,r)$ is a smooth curve of genus $1$, hence the deformation is not $\delta$--constant.  On the other hand,
the deformation $F_s=x^3-y^2+sx^2$ is $\delta$--constant. In fact, the map $\C\to\C^2$ given by $t\mapsto (t^2-s,t^3-ts)$ is a normalization map for $F_s^{-1}(0)$.
Alternatively, we can use
the numerical argument given in Lemma~\ref{lem:sumofdelta} below.
\end{example}

Properties of $\delta$--constant deformations were studied extensively in \cite{Te}; we refer also to \cite[Section II.2.6]{GLS}.
To explain the origin of the name $\delta$--constant, let us recall the following result due to Teissier;  see \cite[Theorem II.2.54]{GLS}.

\begin{lemma}\label{lem:sumofdelta}
Suppose that the singularity of $F_0$ at $z$ has $\delta$--invariant equal to $\delta_0$. Let $s$ be such that $|s|<\rho$ and let $z_1,\ldots,z_n$
be singular points of $F_s^{-1}(0)$. Let $\delta_1,\ldots,\delta_n$ be their $\delta$--invariants. Then the genus of $F_s^{-1}(0)\cap B(z,r)$
is equal to
\[\delta_0-\sum_{j=1}^n\delta_j.\]
\end{lemma}
In particular, the genus is $0$ if and only if $\sum_{j=1}^n\delta_j=\delta_0$. In other words, the sum of $\delta$--invariants
over all singular points of $F_s^{-1}(0)\cap B(z,r)$ does not depend on $s$ for $s$ sufficiently small.

\subsection{Some homological algebra}
For the convenience of the reader, we present a few  facts from homological algebra. 
We refer to \cite{Dold} for more details. 
Let $X$ be a 4--manifold such that $H_1(X)=0$. 

\begin{lemma}\label{lem:det1}
Let $\A$ be an integer matrix representing the intersection form  on $X$, $  H_2(X;\Z)\times H_2(X;\Z)\to\Z$, with respect to some  choice of basis. 
Then $|H_1(\p X)|=|\det \A\, |$.
\end{lemma}
\begin{proof}
From an abstract point of view, the intersection pairing is a composition
\[H_2(X;\Z)\to H_2(X,\p X;\Z)\to H^2(X;\Z)\to\Hom(H_2(X;\Z),\Z).\]
The first map is from the  long exact sequence of the pair, the second
is from Poincar\'e duality and the third   is from the universal coefficient theorem. The second and the third maps are isomorphisms.
If we choose a basis of $H_2(X;\Z)$ and let $\A$ be the corresponding intersection matrix, then the cardinality of
\[\coker \left(H_2(X;\Z)\to \Hom(H_2(X;\Z),\Z)\right)\]
is given by $| \det \A \,|$. But this cokernel is  the cokernel of the map
\[H_2(X;\Z)\to H_2(X,\p X;\Z),\]  and this cokernel is isomorphic to  $H_1(\p X)$, a fact that follows from the long exact homology sequence of pair $(X,\p X)$ and the condition that $H_1(X)=0$.
\end{proof}
\begin{corollary}\label{cor:det1}
Let $\alpha_1,\dots,\alpha_n$ be elements of $H_2(X;\Z)$ that span $H_2(X;\Q)$. Let $B$ be the matrix of intersections $\alpha_i\cdot\alpha_j$. Then
$\alpha_1,\dots,\alpha_n$ span $H_2(X;\Z)$ if and only if $| \det B\, | =  |H_1(\p X;\Z)|$.
\end{corollary}

The intersection form $H_2(X;\Q)\times H_2(X;\Q)\to\Q$ gives rise to the intersection form $H^2(X;\Q)\times H^2(X;\Q)\to\Q$. We can
restrict it to a map
\[H^2(X;\Z)\times H^2(X;\Z)\to\Q,\]
which we will still call the intersection form.

\begin{lemma}\label{lem:det2}
Let $\alpha_1,\dots,\alpha_n$ be a basis of $H_2(X;\Z)$, and let  $\beta_1,\ldots,\beta_n$ the dual basis of $H^2(X;\Z)$. If the intersection form on $H_2(X;\Z)$
is represented by a matrix $\A$ in the given basis, then the intersection form on $H^2(X;\Z)$ is represented by $\A^{-1}$.
\end{lemma}
\begin{proof}
As in the proof of Lemma~\ref{lem:det1}, identify $H^2(X;\Q)$ with $\Hom(H_2(X;\Q),\Q)$. In the bases $\alpha$ and $\beta$, the map
$H_2(X;\Q)\to \Hom(H_2(X;\Q),\Q)$ is given by the matrix $\A$. The pairing $H^2(X;\Q)\to\Hom(H_2(X;\Q),\Q)$ is given by the composition
\[H^2(X;\Q)\xrightarrow{\A^{-1}} H_2(X;\Q)\xrightarrow{\;\A\;}\Hom(H_2(X;\Q),\Q)\xrightarrow{\A^{-1}}\Hom(H^2(X,\Q),\Q),\]
where the last map is dual to $H^2(X,\Q)\to H_2(X,\Q)$. Notice that all the maps are invertible over $\Q$.
Since the intersection form on $H^2(X;\Z)$ is given by the restriction, this concludes the proof.
\end{proof}

\section{Positively self--intersecting concordances and $\delta$--constant deformations}\label{sec:positively}
\subsection{Positively self-intersecting  concordances}\label{sec:psic}
Studying $\delta$-constant deformations of cuspidal singular points is closely related to computing a Gordian distance of two knots or asking 
whether one knot belongs to a minimal unknotting sequence of another; see, for instance,  \cite{Fe}. However, it is not completely the same
question as arises in studying  deformations, since  deformations involve crossing changes combined with concordance. The following notion will
take both into account, and we believe that it is a good translation of the notion $\delta$--constant deformation to a topological setting.
It is related to \cite[Definition 2.1]{CG}.
\begin{definition}
Given two oriented knots $K_0$ and $K_1$ in $S^3$,   a {\em positively self-intersecting concordance} from $K_0$ to $K_1$ is  a smoothly 
immersed annulus $A$  in $S^3\times [0,1]$ for which    $\partial(S^3 \times [0,1],A) = -(S^3, K_0) \cup (S^3,K_1)$ and such that 
the image has only positive self-intersections as singularities; at each singular point, only two branches of $A$ are allowed to intersect
and their intersection must be transverse. The \emph{\level} of a positively self-intersecting concordance is the number of self-intersections of $A$.
We require that this number is finite.
\end{definition}
\begin{example}\label{ex:cogo}
Suppose $K$ can be unknotted by changing $k$ positive crossings to negative. Then 
there is a positively self-intersecting concordance from the unknot to  $K$ with  \level{}~$k$.  The concordance is formed by removing an appropriate embedded 4--ball $B$ from $B^4$ and identifying $B^4-B$ with $S^3 \times [0,1]$ via an orientation preserving homeomorphism carrying $S^3 \times \{1\}$ to $S^3 = \partial B^4$.
\end{example}

Notice that unlike the standard notion of concordance,  one cannot use  positively self-intersecting concordances to generate a symmetric relation; 
the existence of such a concordance from   $K_0$  to $K_1$ does not imply the existence of one from   $K_1$  to $K_0$.

We will now show that a $\delta$--constant deformation gives rise to a positively self-intersecting concordance.

\begin{lemma}\label{lem:ktpsikc}
Let $F_s$ be a $\delta$--constant deformation as in Section~\ref{sec:deform}, such that the singularity of $F_0^{-1}(0)$ is unibranched.
Let $\kc$ be the link of the singularity of $F_0^{-1}(0)$. Let $s$ be close to $0$ and $z_1,\ldots,z_m$ be some (not necessarily all) of the unibranched singular
points of $F_s^{-1}(0)$. Let $\kz{1},\ldots,\kz{m}$ be their links. Let $\kt=\kz{1}\# \kz{2}\ldots\# \kz{m}$.
Then there exists a positively self-intersecting concordance from $\kt$  to $\kc$.
\end{lemma}
\begin{proof}
Consider a ball $B$ as constructed in Section~\ref{sec:deform}. Let $C_s=F_s^{-1}(0)\cap B$. Then $C_s\cap\p B$ is the link $\kc$. Choose a connected arc
$\gamma$
on $C_s$, smooth away from $z_2,\ldots,z_{m-1}$, 
having ends at points $z_1$ and $z_m$ and passing through $z_2,\ldots,z_{m-1}$ (if $m=1$ we choose $\gamma$ to be the point
$z_1$). Let $U$ be a tubular neighbourhood of $\gamma$~in~$B$. If $U$ is sufficiently small, the following hold:

\begin{itemize}
\item $U$ is a four ball with smooth boundary;
\item $\p U\cap C_s$ is isotopic to $\kt$;
\item $A:=\textrm{closure of }C_s\cap(B\sm U)$ is an analytic image of an annulus.
\end{itemize}
If $A$ has only ordinary double points as singularities, then the proof is complete. Otherwise we need to perturb $C_s$. We
will use a well known fact that a generic analytic map from $\C$ to $\C^2$ has only ordinary double points as singularities. Let $\wC\subset\C$ 
be the unit disk, and let $\psi\colon \wC\to C_s\subset\C^2$ be the normalization.

Let us now choose $\phi\colon \wC\to\C^2$ to be a generic analytic function. Fix $\e$ sufficiently small. The function $\wt{\psi}=\psi+\e\phi$
has the following properties.
\begin{itemize} 
\item The image $\wt{\psi}(\wC)\cap \p U$ is isotopic to $\kt$;
\item The image $\wt{\psi}(\wC)\cap \p B$ is isotopic to $\kc$;
\item The inverse image $\wt{\psi}^{-1}(U)$ is a disk and $\wt{\psi}^{-1}(B\sm U)$ is an annulus;
\item $\wt{\psi}(\wC)$ has only ordinary double points as singularities.
\end{itemize}
The first three properties hold if $\e$ is sufficiently small, the last one if $\phi$ is sufficiently generic.
Then $A'=\wt{\psi}(\wC)\cap (B\sm U)$ is a positively self-intersecting concordance from $\kt$ to $\kc$.
\end{proof}
\begin{remark}
The \level{} of the  positively self-intersecting concordance constructed in Lemma~\ref{lem:ktpsikc} 
is equal to $\delta_0-\sum_{j=1}^m\delta_m$. Indeed, each singular point $z$ of $A$ breaks up
into $\delta_z$ ordinary double points in $A'$ (hence,  sometimes the $\delta$--invariant is called the number of double points hidden at the given singular point;
compare \cite[Section~10]{Milnor-sing}).
\end{remark}

\subsection{Positively self-intersecting  concordances and 4-manifolds}\label{sec:psiconcordance}
We begin with the following result.
\begin{proposition}\label{prop:Wconstruct}
Suppose there exists a positively self-intersecting concordance from $K_0$ to  $K_1$ with \level~$p$.
Then for any $q\in\Z$ such that $q\neq 0,-4p$, 
there exists a smooth $4$--manifold $W_q$ with negative definite
intersection form, such that $W_q$ is a cobordism between $S^3_{q}(K_0)$ and $S^3_{q+4p}(K_1)$.
\end{proposition}
\begin{remark}
In this paper we restrict ourselves to integer surgeries, although the approach can be generalized to rational surgeries as well. 
In the applications we will assume that $q$ is a large positive odd integer.
\end{remark}
\begin{proof}
Blow up all singular points of $A\subset S^3\times[0,1]$ (we perform standard, negative, blow-ups). We denote the resulting $4$--manifold  $X'$, 
which is diffeomorphic to $(S^3\times[0,1])\#\ol{\C P^2}\#\ldots\#\ol{\C P^2}$. Let $\wt{A}$ be the strict transform of $A$. It is a smoothly
embedded annulus in  $X'$.

To $X'$ we glue a $2$--handle along $K_1\subset S^3\times\{1\}$ with framing $q+4p$. In this way we obtain a manifold $X_q$ with boundary $S^3_{q+4p}(K_1)\#(-S^3)$.
The core of the $2$--handle with $\wt{A}$ added  forms a smooth surface $S_q\subset X_q$ with boundary $K_0$. Let $N_q$ be a tubular neighborhood of $S_q$ in $X_q$.
We define
\[W_q=\ol{X_q\sm N_q}.\]
The boundary of $W_q$ consists of $S^3_{q+4p}(K_1)$ on one end and  surgery along $K_0$ on the other. The surgery coefficient is the self-intersection
of $S_q$, which is equal to $q+4p-4p=q$, the term $-4p$ comes from the $p$ double points of $A$ (the self-intersection of $\wt{A}$ is $-4p$).
\end{proof}

We now  study homological properties of $W_q$. In what follows we omit the subscript $q$ and write $W$ for $W_q$, $X$ for $X_q$ and $S$ for $S_q$.
It is clear that $H_2(X)=\Z^{p+1}$, $\pi_1(X)=H_1(X)=0$, and  $H_3(X)=\Z$. The second 
homology of $X$ is spanned by the exceptional 
curves $E_1\cup\ldots,E_p$ of the blow-ups
and a surface $E_0$, which can be defined as a union of $S$ and a Seifert surface for $K_0$. These elements are linearly independent,
so they span $H_2(X;\Q)$. To show that
these elements actually span $H_2(X;\Z)$ we observe that $E_0^2=q$, $E_0\cdot E_i=2$, $E_i^2=-1$; a calculation then shows that  
the intersection matrix of the space spanned by
$E_0,E_1,\ldots,E_p$ has determinant $q+4p$, which up to a sign is equal to $|H_1(\p X)|$.  
By Corollary~\ref{cor:det1}, $E_0,E_1,\ldots,E_p$
generate $H_2(X;\Z)$.

The long exact sequence of the pair $(X,W)$ yields
\[H_3(X,W)\to H_2(W)\to H_2(X)\to H_2(X,W)\to H_1(W)\to 0.\]
By excision, we have $H_*(X,W)=H_*(N,\p N)=H_*(D^2,S^1)$. Hence $H_3(W)=\Z$ and $H_2(X,W)=\Z$. The map $H_2(X)\to H_2(X,W)\cong\Z$ can
be explicitly described.  Any $z \in H_2(X)$ represented by a surface $Z$ intersecting $S$ transversely is mapped to $Z\cdot S$ times the generator.    Thus each $E_i$, $i>0$, is mapped to twice the generator. Furthermore, $E_0$ gets mapped to $q$ times the generator. This proves the following
result. 

\begin{lemma}
If $q$ is odd, then $H_2(X)\to H_2(X,W)$ is surjective.
\end{lemma}
From now on, we shall assume that $q$ is odd. Then
\[H_2(W)\cong \Z^p,\ \ H_1(W)\cong 0.\]
The classes $\alpha_j=E_j-E_{j+1}$ for $j=1,\ldots,p-1$ and $\alpha_0=2D-qE_1$ all intersect $S$ trivially, hence they
belong to the kernel of $H_2(X)\to H_2(X,W)$.
In particular they can be regarded as classes in $H_2(W)$.
We claim that these classes generate $H_2(W)$. For dimensional reasons, they 
generate $H_2(W;\Q)$. We  write the intersection matrix of $\alpha_0,\ldots,\alpha_{p-1}$ as
\begin{equation}\label{eq:matrixA}
\mathcal{A}=
\bp
-q^2&q&0&0&\ldots &0\\
q&-2&1&0&\ldots &0\\
0&1&-2&1&\ldots &0\\
0&0&1&-2&\ldots &0\\
\vdots&\vdots&\vdots&\vdots&\ddots&\vdots\\
0&0&0&0&\ldots&-2
\ep
\end{equation}
The determinant is easily seen to be $(-1)^pq(q+4p)$. Corollary~\ref{cor:det1} implies that $\alpha_0,\ldots,\alpha_p$ is a basis of $H_2(W;\Z)$.

We describe the dual basis in $H^2(W)$.  Let $Z'\subset X$ denote the cocore of the $2$--handle and let $Z=Z'\sm N$. 
We let $\beta_0=E_1+E_2+\ldots+E_p-\frac{q+4p-1}{2}Z$
and $\beta_j=E_j+E_{j+1}+\ldots+E_p$ 
for $j=1,2,\ldots,p$. It is clear that $\beta_i\cdot \alpha_j=1$ if $i=j$ and $0$ otherwise. In other words, $\beta_0,\ldots,\beta_p$
is the dual basis for $\alpha_0,\ldots,\alpha_p$. By Lemma~\ref{lem:det2}, 
the intersection map $H^2(W)\times H^2(W)\to\Q$ is given with respect to the  basis $\beta_0,\ldots,\beta_p$
by $\mathcal{A}^{-1}$. In particular
the self-intersection of $\beta_0\in H^2(W)$ is  
\be\label{eq:bo2}
\beta_0^2=\frac{-p}{q(q+4p)}.
\ee

\subsection{\spinc{} structures on $W_q$ and $S^3_{q}(K_0)$, $S^3_{q+4p}(K_1)$}\label{sec:spinc}
We begin by recalling the convention of \cite[Section 3.4]{OzSz04} regarding \spinc{} structures on surgeries on knots. Let $J\subset S^3$
be a knot and $d\in\Z_{>0}$. Let $M=S^3_d(J)$ and let $\O$ be a $4$--manifold arising after adding a $2$--handle to a ball $B^4$ along $J$ with
framing $d$, so that $\p\O=M$. Let $\Sigma$ be the generator of $H_2(\O;\Z)$. We can regard $\Sigma$ as a union of a Seifert surface for $J$ with the core
of the $2$--handle.
\begin{definition}\label{def:m}
For $m\in[-d/2,d/2)$, the \spinc{} structure $\sss_m$ on $M$ is the unique \spinc{} structure that  extends over $\O$ to a \spinc{} structure $\sst_m$
satisfying $\langle c_1(\sst_m),\Sigma\rangle +2m=d$.
\end{definition}

Consider the following commutative diagram
\[\xymatrix{
H^2(\O;\Z)\ar[d]^{\simeq}\ar[r]& H^2(M;\Z)\ar[d]^{\simeq}\\
H_2(\O,\p\O;\Z)\ar[r]&H_1(M;\Z),
}
\]
where the vertical arrows are given by Poincar\'e duality and the horizontal maps are from the long exact sequences of pairs.
The second cohomology group $H^2(\O;\Z)$ is generated by the dual to the class $[Z]$ of the cocore of the $2$--handle.
The image of $[Z]$ in $H_1(M;\Z)$ is the meridian $\mu$ of the knot $J$.

Suppose that $d$ is odd, so $H^2(M)$ has no 2--torsion. The first Chern class determines  a unique \spinc{} structure on $M$.
For any $m\in[-d/2,d/2)$ we have $c_1(\sst_m)=(d-2m)[Z]$. Its image in $H^2(M;\Z)$ is Poincar\'e dual to $(d-2m)\cdot\mu$.
Since the map $H^2(\O;\Z)\to H^2(M;\Z)$ is surjective, any element of type $k\mu\in H_1(M;\Z)$ for $k\in\Z$ determines
a \spinc{} structure $\sst_m$ on $\O$ for some $m$. More precisely, we have the following lemma.

\begin{lemma}\label{lem:meridianspin}
For any $k\in\Z$, the class $k\mu\in H_1(M)$ determines the \spinc{} structure $\sss_m$, where $m\in[-d/2,d/2)$ and    $(d-2m)\equiv k\bmod d$.
\end{lemma}

\smallskip
We return to considering  $W_q$. For any odd $r\in\Z$, the class
\begin{equation}\label{eq:cr}
c_r=r\beta_0\in H^2(W_q)
\end{equation}
is    characteristic, so it
determines a \spinc{} structure on $W_q$, denoted $\sst_r$, such that $c_1(\sst_r)=c_r$. 
We now compute the restriction of $\sst_r$ to $S^3_q(K_0)$ and $S^3_{q+4p}(K_1)$.
Geometrically we have
\begin{align*}
\beta_0\cap S^3_{q+4p}(K_1)&=-\frac{q+4p-1}{2}\mu_1\\
\beta_0\cap S^3_{q}(K_0)&=-\frac{q-1}{2}\mu_0,
\end{align*}
where $\mu_1$ and $\mu_0$ are (images of) meridians of $K_1$ and $K_0$ in $S^3_{q+4p}(K_1)$ and $S^3_{q}(K_0)$, respectively. Strictly speaking,
we should write $\beta_0\cap -S^3_q(K_0)=-(-\frac{q-1}{2})\mu_0$, because the orientation of $S^3_q(K_0)\subset W_q$ is reversed, but the result is the same.

Combining the   observation above with Lemma~\ref{lem:meridianspin}, we obtain the following relation.
\begin{lemma}\label{lem:sss}
If $r$ satisfies the following constraints:
\begin{itemize}
\item $q\equiv r\bmod 4$ 
\item $\frac{q+4p-r}{4}\in[-(q+4p)/2,(q+4p)/2)$
\item $\frac{q-r}{4}\in[-q/2,q/2)$, 
\end{itemize} 
then $\sst_r$ restricts to the class $\sss_{m_1}$ on $S^3_{q+4p}(K_1)$
and $\sss_{m_0}$ on $S^3_{q}(K_0)$ where
\begin{align*}
m_1&=\frac{q-r}{4}+p\\
m_2&=\frac{q-r}{4}.
\end{align*}
\end{lemma}

\subsection{$d$--invariants}\label{sec:d}

\noindent  Proposition~\ref{prop:d-compute}, together with the construction of the manifold $W_q$, gives the following general result.
\begin{theorem}\label{prop:dinv}
Suppose that $K_0$ and $K_1$ are connected sums of algebraic knots and there exists a positively self-intersecting concordance from  $K_0$  to $K_1$ with \level~$p$. Then for any $m\in\Z$ we have
\[I_{{K_1}}(m+g(K_1)+p)\le I_{K_0}(m+g(K_0)),\]
where $I$ is the gap function defined in Section~\ref{sec:surg} and $g(K)$ denotes the three genus of the knot $K$.
\end{theorem}
\begin{proof}
Choose $m\in\Z$ and fix $r$ large enough and odd so that $4m+r>2g(K_0)$ and  $4m+r+4p>2g(K_1)$. We set $q=4m+r$. We may assume that $|m|<q/2$, otherwise
we can further  increase  $r$ and $q$. By Lemma~\ref{lem:sss}, 
the \spinc{} structure $\sst_r$ on $W_q$ restricts to the \spinc{} structure $\sss_{m+p}$ on $S^3_{q+4p}(K_1)$
and $\sss_m$ on $S^3_q(K_0)$. We now  use  the fundamental inequality between $d$--invariants stated in Proposition~\ref{prop:fundbound}:
\[d(S^3_{q+4p}(K_1),\sss_{m+p})-d(S^3_{q}(K_0),\sss_m)\ge \frac{c_1(\sst_r)^2-2\chi(W_q)-3\sigma(W_q)}{4}.\]
Since $c_1(\sst_r)=r\beta_0$, we have by \eqref{eq:bo2}
\[c_1(\sst)^2=\frac{-r^2p}{q(q+4p)}.\]
Moreover, $\chi(W_q)=p$ and $\sigma(W_q)=-p$. Substituting \eqref{eq:dofsurg} we obtain
\begin{multline*}
\left(-2I_{K_1}(m+g(K_1)+p)+\frac{(-2(m+p)+q+4p)^2-(q+4p)}{4(q+4p)}\right)-\\
-\left(-2I_{K_2}(m+g(K_0))+\frac{(-2m+q)^2-q}{4q}\right)\ge\\
\ge\frac{-r^2p}{4q(q+4p)}+\frac{p}{4}.
\end{multline*}
A straightforward calculation shows that the inequality simplifies to
\[-2I_{K_1}(m+g(K_1)+p)+2I_{K_0}(m+g(K_0))\ge 0.\]
\end{proof}

\begin{remark}\label{rem:smooth}
Before considering applications of these results to deformations of singular points, we observe that this  simple inequality belongs to the smooth world: it depends on  the fact that $A$ is a smoothly immersed annulus, or equivalently, that $W_q$ is a smooth, not just a topological, manifold.
Take $K_0$ to be the unknot, $K_1$ to be an algebraic knot of genus $g=g(K_1)$, and let  $m=0$.
Theorem~\ref{prop:dinv} implies that if there exists a positively self-intersecting concordance from  $K_0$   to $K_1$ with \level~$p$, then
\[I_{K_1}(g+p)=0.\]
The largest element that is  not in the semigroup is $2g-1$; see Lemma~\ref{lem:semiprop}. It follows that $I_{K_1}(2g-1)=1$, so $p>g-1$.
If we combine this with Example~\ref{ex:cogo}, then it follows that an algebraic knot cannot be unknotted in less than $g$ unknotting moves;
this result  was first proved by Kronheimer and Mrowka \cite{KM} and Heegaard Floer proofs are well-known~\cite{Liv, OS,Sar}.  

Note that there are algebraic knots (in fact torus knots) that bound topological locally flat immersed disks in $B^4$ with the number of double points less than the unknotting number of the knot. In fact, one can show that algebraic unknotting number (see \cite{Sae} for definition) of some torus knots is actually less than
the unknotting number.  It is also known that the topological four genus of torus knots is sometimes less than the smooth four genus; see \cite[Section 5]{Rud}
and references therein.  This shows the value of working in the smooth category, as opposed to topological locally flat category, in studying deformations of singularities from a topological perspective.
\end{remark}

We now prove the main  theorem of the present article.

\begin{theorem}[Semicontinuity of semigroups]\label{thm:main}
Assume that there exists a $\delta$--constant deformation of a cuspidal singular point $z_0$, such that the general fiber
has among its singularities cuspidal singular points $z_1,\ldots,z_n$. Let $S_0,S_1,\ldots,S_n$ be the corresponding semigroups.
Then for any $m\ge 0$ and for any nonnegative $m_1,\ldots,m_n$ such that $m_1+\ldots+m_n=m$, we have
\[\# S_0\cap [0,m)\le \sum_{j=1}^n\# S_j\cap [0,m_j).\]
\end{theorem}
\begin{proof}
Let $\kc$ be the link of singularity at $z_0$ and let $K_1,\ldots,K_n$ be the links of singular points $z_1,\ldots,z_n$. Let $\kt=K_1\#\ldots\# K_n$.
By Lemma~\ref{lem:ktpsikc}, there is a positively self-intersecting concordance from $\kt$ to $\kc$. The \level{} of that concordance is equal to $\delta_0-\sum\delta_j$,
which is equal to $g(\kc)-g(\kt)$. Theorem~\ref{prop:dinv} applied to $-m+g(\kt)$ implies that
\[I_{\kc}(-m+2g(\kc))\le I_{\kt}(-m+2g(\kt)).\]
From the definition of the infimum convolution, using the fact that $g(\kt)=\sum g(K_j)$, we obtain
\[I_{\kc}(-m+2g(\kc))\le \sum_{j=1}^n I_{K_j}(-m_j+2g(K_j)).\]
But now, by Lemma~\ref{lem:semiprop}(d), we have $I_{K_j}(-m_j+2g(K_j))=\# S_j\cap[0,m_j)$ and $I_{\kc}(-m+2g(\kc))=\# S_0\cap[0,m)$.
\end{proof}

\subsection{Minimal unknotting sequences of torus knot}\label{sec:baader}
Sebastian Baader pointed out the  application presented in this section. Unknotting sequences of torus knots were investigated in \cite{Ba0,Ba}. 

\begin{theorem} \label{thm:unknot} If  the  torus knot $T(a,b)$ with $0< a<b$ appears in some minimal unknotting sequence for the torus knot $T(c,d)$  with $0< c<d$, then $a \le c$.
\end{theorem}

\begin{proof}  Suppose that $T(a,b)$ is in such an unknotting sequence.  Then there is a singular concordance from  $T(a,b)$ to $T(c,d)$  having all positive double points.  The minimal number of crossing changes in an unknotting sequence for any torus knot $K$ is $g(K)$.  Thus, the  double point count is $ p = g(T(c,d)) - g(T(a,b))$.

  Let $K_0=T(a,b)$ and $K_1=T(c,d)$.  According to   Theorem~\ref{prop:dinv},  $$I_{K_1}(m+g(K_1)+p)\le I_{K_0}(m+g(K_0)).$$ 
  Replacing $m$ in this equation with $-m + g(K_0)$  and  using the fact that $p + g(K_0) = g(K_1)$, we have
  $$I_{K_1}(-m+2g(K_1) )\le I_{K_0}(-m+2g(K_0)).$$
  
As in the proof of Theorem~\ref{thm:main}, Lemma~\ref{lem:semiprop}(d)  can be applied to show that this inequality implies  that 
for any $m\in\Z$,
$$\#S_{c,d}\cap[0,m)\le\#S_{a,b}\cap[0,m),$$
where $S_{c,d}$ and $S_{a,b}$ are the semigroups associated to singular points with links $T(c,d)$ and $T(a,b)$, respectively.
  The value   $ \#S_{c,d} \cap [0,c+1) =2$;  $\#S_{a,b}\cap[0,c+1)\ge 2$ if and only if $a \le c$.  This completes the proof.  
\end{proof} 

\section{Examples and application}\label{examplesection}
\subsection{Semicontinuity of the semigroup and semicontinuity of the spectrum}\label{sec:ssandss}
The word ``semicontinuity" in singularity theory is  often used in the context of the semicontinuity of the spectrum, see \cite{St,Var}. A natural question
that arises is  whether the semicontinuity of semigroups is related to the semicontinuity of the spectrum. In general, the semicontinuity of the spectrum
is a stronger obstruction. For example,  fix coprime integers $p,q$ and   ask  what is the maximal possible $k$ such that the $(p;q)$ singularity
can be deformed (perturbed) 
to a $(2;2k+1)$ singularity. We refer to this question as the Christopher--Lynch problem, \cite{BZ3,ChLy}, and the conjecture
of \cite{ChLy} is that $k\le p+q-\intfrac{q}{p}-3$. The Milnor number gives $k\le\sim pq/2$, while the spectrum that $k\le \sim pq/4$.
Unfortunately, the semicontinuity of the semigroup also gives $k\le\sim pq/2$ (the symbol $\sim$ denotes an asymptotic estimate up
to linear terms in $p$ and $q$).

We  next present an example which shows that sometimes the semicontinuity of the semigroup provides an obstruction to the existence
of a given deformation (using Theorem~\ref{thm:main}, a $\delta$--constant deformation; \cite[Section 3.5]{GN} would obstruct
any analytic deformation), when
the semicontinuity of the spectrum fails to give an obstruction.
 
We ask  whether a singularity with characteristic sequence $(6;7)$ can be deformed (perturbed) to a singularity with characteristic sequence $(4;9)$
(in the notation from Section~\ref{sec:deform} the singularity of $F_0^{-1}(0)$ is supposed to be $(6;7)$; we ask if we can find $F_s$
such that $F_s^{-1}(0)$ has singularity $(4;9)$).
The Milnor numbers of these singularities are respectively $30$ and $24$. Below we shall discuss which criteria obstruct the existence of this such deformation. 

\smallskip
\textbf{A. The semigroup semicontinuity property.}
The semigroups $S_{6,7}$ for $T(6,7)$ is generated by $6$ and $7$; the semigroup $S_{4,9}$ of $T(4,9)$ is generated by $4$ and $9$. Consider $m=8$. We
have $S_{6,7}\cap [0,8)=\{0,6,7\}$ and $S_{4,9}\cap [0,8)=\{0,4\}$. Thus Theorem~\ref{thm:main} obstructs the existence of a $\delta$--constant
deformation.

\vskip-5mm
\begin{remark}
While Theorem~\ref{thm:main} obstructs the existence a $\delta$--constant deformation, the result of Gorsky and N\'emethi obstructs the existence of
any deformation. On the other hand, neither  result obstructs the existence of a smooth genus $3$ cobordism betweeen $T(6,7)$ and $T(4,9)$, because 
the results of Gorsky and N\'emethi hold only in the analytic case, while our theorem obstructs the existence of a positively self--intersecting concordance
with double point count $3$.
\end{remark}

\vskip -5mm
\textbf{B. The semicontinuity of the spectrum.}
The spectra of $x^6-y^7=0$ and $x^4-y^9=0$ can be computed from a general formula \cite{Sait}
or from the Thom--Sebastiani formula;  see \cite[Theorem 5.31]{Zo}. In \eqref{eq:spec} we present the spectra in the following joint list of rational numbers. The elements in
the spectrum of $x^6-y^7=0$ are written in normal fonts, and the element  of the spectum $x^4-y^9=0$ are written in boldface. 
It is a straightforward calculation using this list that each interval of type\nopagebreak[4] $(x,x+1)$,
\nopagebreak[4] $x\in\R$, contains at least as many non-boldfaced elements as boldfaced.
In other words, the spectrum does not obstruct the existence of deformation of $\{x^6-y^7=0\}$ to $\{x^4-y^9=0\}$.

\goodbreak
\begin{equation}\label{eq:spec}
\begin{array}{rrrrrrrr}
13/42 &  \textbf{13/36} &  19/42 &  \textbf{17/36} &   10/21 &   \textbf{7/12} &   25/42 &   \textbf{11/18} \\   
13/21 & 9/14 &  \textbf{25/36} & \textbf{13/18} &   31/42 &   16/21 &   11/14 &   \textbf{29/36} \\ 
17/21  & \textbf{5/6} &   \textbf{31/36} &   37/42 &   19/21 &   \textbf{11/12} & 13/14 &   \textbf{17/18} \\   
20/21 & \textbf{35/36} &   41/42 &   43/42 &   \textbf{37/36} &   22/21 &   \textbf{19/18} &   15/14 \\
\textbf{13/12} & 23/21 &   47/42 &   \textbf{41/36} &   \textbf{7/6} &  25/21 &   \textbf{43/36} &   17/14 \\
26/21 &   53/42 &   \textbf{23/18} &   \textbf{47/36} & 19/14 &  29/21 &   \textbf{25/18} &  59/42 \\
\textbf{17/12} & 32/21 &  \textbf{55/36} &  65/42 &  \textbf{59/36} &  71/42.
\end{array}
\end{equation}

\smallskip
\textbf{C. The $\Upsilon$ function of Ozsv\'ath, Szab\'o and Stipsicz.}

In \cite{OSS} a function called $\Upsilon_K(t)\colon[0,2]\to\R$ was associated to any knot $K\subset S^3$. It is derived from the knot Heegaard Floer homology
and it is shown to be a smooth concordance invariant. It satisfies the symmetry property $\Upsilon_K(2-t)=\Upsilon(t)$; for ease of exposition we will restrict to
the interval $[0,1]$.

The results of \cite{OSS} imply that $\Upsilon$ can be used to study positive self--intersecting concordance. More precisely, 
\cite[Proposition 1.10 and Theorem 1.11]{OSS} yield the following result.

\begin{theorem}\label{thm:oss1}
Let $K_0$ and $K_1$ be two knots in $S^3$. Suppose there is a positively self--intersecting concordance from $K_0$ to $K_1$ with double point count $p$. Then,
for any $t\in[0,1]$ we have.
\[\Upsilon_{K_1}(t)\le \Upsilon_{K_0}(t)\le \Upsilon_{K_1}(t)+p\cdot t.\]
\end{theorem}

Another feature of $\Upsilon$ is that for algebraic knots (more generally, for all $L$--space knots) it can be computed from the Alexander polynomial. 
The computations in \cite[Theorem~1.15]{OSS} can be easily reformulated in terms of the semigroup of the singular point. The resut is as follows.

\begin{proposition}\label{prop:oss2}
Let $K$ be an algebraic knot with genus $g$ and $S$ the corresponding semigroup. Then for any $t\in[0,1]$ we have
\[\Upsilon_K(t)=-2\!\!\!\!\!\!\!\min_{m\in\{0,\ldots,2g\}}\left(\#S\cap[0,m)+\frac{t}{2}(g-m)\right).\]
\end{proposition}

For the knots $K_0=T(4,9)$ and $K_1=T(6,7)$ we easily compute
\[
\Upsilon_{T(4,9)}(t)=\begin{cases} -12t& t\in[0,\frac12]\\ -4t-4&t\in[\frac12,1]\end{cases}\;\;\;\;\;
\Upsilon_{T(6,7)}(t)=\begin{cases} -15t& t\in[0,\frac13]\\ -9t-2&t\in[\frac13,\frac23]\\ -3t-6&t\in[\frac23,1].\end{cases}
\]

Therefore we have
\[\delta(t):=\Upsilon_{T(4,9)}(t)-\Upsilon_{T(6,7)}(t)=
\begin{cases}
3t&t\in[0,\frac13]\\
-3t+2&t\in[\frac13,\frac12]\\
5t-2&t\in[\frac12,\frac23]\\
-t+2&t\in[\frac23,1].
\end{cases}
\]
We have $0\le \delta(t)\le 3t$ for $t\in[0,1]$; see Figure~\ref{fig:deltagraph}. The conclusion is that the Ozsv\'ath--Szab\'o--Stipsicz
criterion \emph{does not} obstruct the existence of a positively self--intersecting concordance from $T(4,9)$ to $T(6,7)$ with double point count $3$.
That is, it \emph{does not} obstruct the existence of a $\delta$--constant deformation of a singularity $(6;7)$ to a singularity $(4;9)$.
It follows that the  $\Upsilon$ function gives in general \emph{different} obstructions than the $d$--invariants of large surgeries.

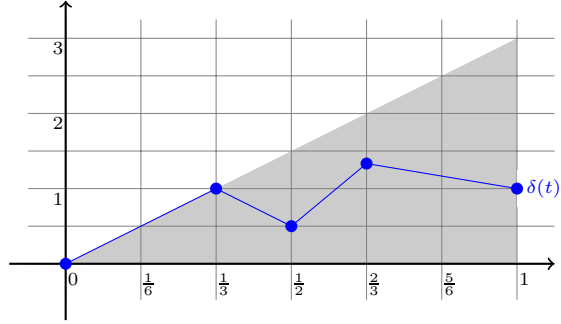
\begin{figure}
\begin{tikzpicture}
\fill[black!20!white] (0,0)--(6,0) -- (6,3) -- (0,0);
\draw[xstep=1,ystep=0.5 ,gray, very thin](-0.5,-0.5) grid (6.5,3.25);
\draw[very thick, white] (-0.5,-0.5) -- (7,-0.5);
\draw[->,thick] (0,-0.75) -- (0,3.5);
\draw[->,thick] (-0.75,0) -- (6.5,0);
\draw[blue] (0,0) -- (2,1) -- (3,0.5) -- (4,1.3333) -- (6,1) node [right, fill=white] {\scriptsize $\delta(t)$};
\fill[blue] (0,0) circle [radius=0.08];
\fill[blue] (2,1) circle [radius=0.08];
\fill[blue] (3,0.5) circle [radius=0.08];
\fill[blue] (4,1.3333) circle [radius=0.08];
\fill[blue] (6,1) circle [radius=0.08];
\draw (0.1,0) node[below] {\scriptsize $0$};
\draw (1.1,0) node[below] {\scriptsize $\frac16$};
\draw (2.1,0) node[below] {\scriptsize $\frac13$};
\draw (3.1,0) node[below] {\scriptsize $\frac12$};
\draw (4.1,0) node[below] {\scriptsize $\frac23$};
\draw (5.1,0) node[below] {\scriptsize $\frac56$};
\draw (6.1,0) node[below] {\scriptsize $1$};
\draw (-0.1,1) node[below=-2pt] {\scriptsize $1$};
\draw (-0.1,2) node[below=-2pt] {\scriptsize $2$};
\draw (-0.1,3) node[below=-2pt] {\scriptsize $3$};
\end{tikzpicture}
\caption{The graph of function $\delta=\Upsilon_{T(4,9)}-\Upsilon_{T(6,7)}$. The shaded region is the region $\{(x,y)\colon 0\le y\le 3x\}$.}\label{fig:deltagraph}
\end{figure}

\smallskip\goodbreak
\textbf{D. Other criteria.}\vskip.05in 

\underline{D1.} In  \cite[Theorem 1.6]{Bosing}, obstructions of a different type than those studied here were developed to obstruct the existence of deformations.  However, one can check that these do not obstruct there being a deformation of $(6;7)$ to $(4;9)$. 
\vskip.05in 

\underline{D2.} A conjectural semicontinuity of the $\ol{M}$-number, see \cite{Bo,Or}, 
says that $\ol{M}_{6,7}$ should be greater than $\ol{M}_{4,9}$; otherwise a deformation does not exists. But 
$\ol{M}_{6,7}=7+6-\intfrac{7}{6}-3=9$, $\ol{M}_{4,9}=4+9-\intfrac{9}{4}-3=8$. So the semicontinuity of $\ol{M}$ numbers would not
obstruct that deformation.\vskip.05in 

\underline{D3.} We can also try to obstruct deformations of $(6;7)$ to $(4;9)$ in an explicit way. We search for a 
deformation of the form $x_s(t)=a_4(s)t^4+a_5(s)t^5+\ldots$, $y_s(t)=b_4(s)t^4+b_5(s)t^5+\ldots$, where $a_i(s)$
and $b_j(s)$ are coefficients depending smoothly on a deformation parameter $s$,
compare \cite[Section 2]{Bosing}, where $a_4(s)$ vanishes only at $s=0$ and $a_5(0)=0$. We require that for $s\neq 0$ the polynomials parametrize a $(4;9)$
singularity, and for $s=0$ they parametrize an $(6;7)$ singularity. We will show that this is impossible. Assuming
that for $s\neq 0$ the singularity has type $(4;9)$, we explicitly write the Puiseux expansion
\[y_s(t)=c_4(s)x_s+c_8(s)x_s^2+c_9(s)x_s^{9/4}+\ldots.\]
Observe that either $a_6(0)\neq 0$, or $a_7(0)\neq 0$, for otherwise the singularity cannot have type $(6;7)$. It follows that 
$c_4(s)$ extends continuously to $0$. In fact, $b_6(s)=c_4(s)a_6(s)$ and $a_6,b_6$ are continuous at $s=0$ (and $b_7(c)=c_4(s)a_6(s)$). 
Thus we may replace $y_s$ by $y_s-c_4(s)x_s$
and assume that $c_4(s)\equiv 0$. It follows that for $s\neq 0$, the function $y_s(t)$ must vanish at $t=0$ up to order at least $8$. Thus $y_0(t)$
must also vanish at $t=0$ up to order at least $8$. This contradicts the assumption that $(x_0(t),y_0(t))$ parametrizes a $(6;7)$ singularity.
This approach relies on the analytic parametrization.

\subsection{Examples with many singular points}\label{sec:many}

Here we provide applications of Theorem~\ref{thm:main} in the case $n\ge 2$, that is, in the case not covered by the result of Gorsky and N\'emethi.

\begin{example}
For any $r>0$, there does not exist a $\delta$--constant deformation in which the central fiber has singularity   type $(6;8,2r+7)$ and the generic fiber has two cuspidal
singularities with multiplicities $4$ and $5$ respectively.

Indeed, if it existed, we choose $m=9$, $m_1=4$, $m_2=5$. 
We let $S_0$  denote the semigroup of the central singularity,   $S_1$  the semigroup of the singularity with multiplicity
$4$,  and   $S_2$   the  semigroup of the singularity with multiplicity $5$. With these choices we have
\[\#S_0\cap[0,m)\ge 3,\]
because $\{0,6,8\}\in S_0$. On the other hand $S_1\cap[0,4)=S_2\cap[0,5)=\{0\}$, so the inequality in Theorem~\ref{thm:main} is violated.
\end{example}

This obstruction is of a different nature than the obstruction coming from semicontinuity of spectra. In fact, if we fix the type of the singular points in the
generic fiber to be $(4,a)$ and $(5,b)$ for some $a,b$ satisfying $\gcd(4,a)=\gcd(5,b)=1$, $a>4$ and $b>5$, then for large $s$, neither the genus bound nor
the spectrum semicontinuity will give an obstruction.

\begin{remark}\label{rem:semimult}
We sketch an elementary analytic argument obstructing the above deformation. Namely, if the deformation is $\delta$--constant, we can find a family of parametrizations
$(x_s(t),y_s(t))$, where $t\mapsto (x_0(t),y_0(t))$ locally parametrizes  a $(6;8,2r+7)$ singularity and  $t\mapsto (x_s(t),y_s(t))$ for $s\neq 0$ parametrizes
a curve having two singular points of types $(4;a)$ and $(5;b)$. Comparing the total number of zeros of the derivative 
$\frac{\partial x_s(t)}{\partial t}$ for $s=0$ and $s\neq 0$,
we arrive at the inequality $(6-1)\ge (4-1)+(5-1)$. This reasoning can be generalized to the case when the deformation is not $\delta$--constant,
but the final inequality is weaker (the genus
of the non-central fiber enters the inequality) and does not obstruct, in general, the deformation.

Notice that the above argument does not work in the smooth setting.
\end{remark}

There are more complicated deformations, which cannnot be obstructed by a ``semicontinuity of the multiplicities'' argument as in Remark~\ref{rem:semimult}.

\begin{example}
Consider a singularity  $(10;12,2r+13)$ for $r>0$. It cannot be deformed to a pair singularities $(5,6)$ and $(5,61)$ for a $\delta$--constant
deformation. In fact, we take $m=61$ in Theorem~\ref{thm:main}. 
One readily computes that the semigroup for $(10;12,2r+13)$ contains at least $21$ elements in range $[0,61)$.
We take $m_1=1$ and $m_2=60$, then there is $1$ element in the semigroup of $(5;6)$ in range $[0,1)$ and $12$ elements in the semigroup of the singularity $(5;61)$
in range $[0,60)$ so the obstruction applies and the deformation does not exist. 

Again for $r$ sufficiently large the semicontinuity of the spectrum does not obstruct the deformation.
\end{example}

\end{document}